\documentclass[11pt]{amsart}
\usepackage{amsmath,amssymb,enumerate,etoolbox,mathrsfs}
\usepackage[all]{xy}
\usepackage{enumitem}
\numberwithin{equation}{section}

\theoremstyle{definition}  

\newtheorem{thm}{Theorem}[section]

\newtheorem{prop}[thm]{Proposition}

\newtheorem{conj}[thm]{Conjecture}

\newtheorem{lem}[thm]{Lemma}
\theoremstyle{remark}

\newcommand{\F}{\mathbb{F}}
\newcommand{\N}{\mathbb{N}}
\renewcommand{\P}{\mathbb{P}}
\newcommand{\Q}{\mathbb{Q}}
\newcommand{\Z}{\mathbb{Z}}

\newcommand{\cA}{\mathcal{A}}

\newcommand{\cC}{\mathcal{C}}

\newcommand{\cU}{\mathcal{U}}
\newcommand{\m}{\mathfrak{m}}

\newcommand{\Gal}{\mathrm{Gal}}
\newcommand{\Spec}{\mathrm{Spec}\,}
\newcommand{\tor}{\mathrm{tor}}
\newcommand{\Tr}{\mathrm{Tr}}

\newcommand{\GL}{\mathrm{GL}}
\newcommand{\Hom}{\mathrm{Hom}}
\renewcommand{\SS}{\mathsf{S}}

\title[Virtual freeness for abelian varieties]{Virtual freeness for abelian varieties over boundedly generated fields}

\author{Bo-Hae Im}
\address{Department of Mathematical Sciences, KAIST, 291 Daehak-ro, Yuseong-gu, Daejeon, 34141, South Korea}
\email{bhim@kaist.ac.kr}

\author{Michael Larsen}
\address{Department of Mathematics, Indiana University, Bloomington, IN, 47405, U.S.A.}
\email{mjlarsen@iu.edu}
\thanks{Bo-Hae Im was supported by the Basic Science Research Program through the National Research Foundation of Korea (NRF) grant funded by the Korean government (MSIT) (NRF-2023R1A2C1002385, or RS-2023-NR076333). Michael Larsen was partially supported by NSF grant DMS-2401098 and the Simons Foundation.}

\begin{document}

\begin{abstract}
Let $\Omega$ be a \emph{boundedly generated} extension of $\Q$; that is, for some $d$ and some finitely generated extension $K$ of $\Q$, $\Omega$ is generated over $K$ by elements of degree $\le d$. If $A/\Omega$ is an abelian variety,
then $A(\Omega)$ is a virtually free abelian group.

As an application, we prove that if $\dim A = 1$ and $\sigma_1,\ldots,\sigma_n\in \Gal(\bar K/K)$, then the rank of $A$ over the invariant subfield $\bar K^{\langle \sigma_1,\ldots,\sigma_n\rangle}$ is infinite.
This settles the $1$-dimensional case of a conjecture of one of us \cite{Larsen} and the genus $1$ case of a conjecture of Junker and Koenigsmann \cite{JK}.

\end{abstract}

\maketitle

\section{Introduction}

We say a field  $\Omega$ is \emph{boundedly generated} over $\Q$ if there exists a subfield $K$ of $\Omega$ such that $K$ is finitely generated over $\Q$ and $\Omega$ is generated over $K$ by a 
(possibly infinite) set of elements, each algebraic of degree $\le d$ over $K$. We are interested in understanding the group of points of an abelian variety $A$ over $\Omega$. We may always assume that the subfield $K$ 
over which $\Omega$ is generated by elements of bounded degree is large enough that $A$ is defined over it. 
Throughout this paper, therefore, we will assume that $K$ is a field finitely generated over $\Q$ and $A$ is an abelian variety defined over $K$.

By N\'eron's extension of the Mordell-Weil theorem \cite[Chapter 6, Theorem 1]{Lang}, the group $A(K)$ is finitely generated. In particular, it contains a finitely generated free abelian subgroup; i.e., it 
is virtually free. Likewise $A(L)$ is finitely generated for all finite extensions $L/K$, but this certainly does not imply that $A(\Omega)$ is finitely generated (and, in general, it is not). 

Silverman proved \cite[Lemma]{Si} that if $K$ is a number field, $d$ is a positive integer, and $A$ is an elliptic curve over $K$ then there is an upper bound on the order of $A(L)_{\tor}$ as $L$ ranges over
extensions of $K$ of degree $\le d$. 
In \cite[Proposition 6]{IL-HJ}, this was extended to higher dimensional abelian varieties $A$ and to finitely generated fields.
These results suggest, but certainly do not imply, that $A(\Omega)_{\tor}$ may be bounded.

For every positive integer $d$, we define $K(d)$ to be the subfield of $\bar K$ generated by all $K$-extensions $L\subset \bar K$
which are of degree $\le d$ over $K$. We prove the following common extension of the theorems of Silverman and N\'eron.
\begin{thm}
\label{main}
The group $A(\Omega)$ is virtually free abelian.
\end{thm}

It suffices to prove this in the case that $\Omega$ is of the form $K(d)$. Indeed, we can always embed $\Omega$ as a subfield of $K(d)$ for some $K$ and $d$.
If $A(K(d))$ is virtually free abelian, then 
$A(\Omega)$ is a subgroup of a virtually 
abelian group. It therefore contains a finite index subgroup which is contained in a free abelian group, so it is itself virtually free abelian.

The Mordell-Weil groups that Silverman and N\'eron considered are finitely generated, but typically $A(K(d))$ is not. To prove it is virtually free, therefore, we need to show both that $A(K(d))_{\tor}$ is finite and that 
$A(K(d))/A(K(d))_{\tor}$ is free abelian. 

As an application of Theorem~\ref{main}, we prove the following result:

\begin{thm}
\label{std}
If $\dim A=1$ and $\sigma_1,\ldots,\sigma_n\in G_K$, then
\begin{equation}
\label{infinite}
\dim_{\Q} A(\bar K^{\langle \sigma_1,\ldots,\sigma_n\rangle})\otimes \Q = \infty.
\end{equation}
\end{thm}

In 1974,  Frey and  Jarden \cite{FJ} proved that for any abelian variety $A/K$,
if $\sigma_1,\ldots,\sigma_n$ are chosen independently from the uniform distribution on $G_K$, then \eqref{infinite} holds with probability $1$.
This has been significantly strengthened by Asayama and Taguchi, who proved \cite[Corollary 4.3]{AT} that 
if $n\ge 2$, with probability $1$, $A(\bar K^{\langle \sigma_1,\ldots,\sigma_n\rangle})$ is isomorphic to
the direct sum of a finite abelian group and a free abelian group of countably infinite rank.

In 2003, one of us \cite{Larsen} conjectured that \eqref{infinite}, in fact, holds in an absolute sense.  Writing $\tilde K$ for $\bar K^{\langle \sigma_1,\ldots,\sigma_n\rangle}$, the assertion is

\begin{conj}
\label{equiv}
If $A$ is any non-trivial abelian variety over any field $\tilde K$ of characteristic zero such that $G_{\tilde K}$ is finitely generated, then $A$ has infinite rank over $\tilde K$.
\end{conj}

The case $\dim A = 1$ of this conjecture follows immediately from Theorem~\ref{std}. Indeed, given $A/\tilde K$,
choose a finitely generated subfield $K$ of $\tilde K$ over which $A$ can be defined. Fix an algebraic closure of
$\tilde K$, and let $\bar K$ denote the algebraic closure of $K$ in this closure. Thus we have a natural homomorphism
$G_{\tilde K}\to G_K$. If $\sigma_1,\ldots,\sigma_n$ denote the images in $G_K$ of a sequence of topological
generators of $G_{\tilde K}$, then $A(\tilde K)\otimes \Q \supset A(\bar K^{\langle \sigma_1,\ldots,\sigma_n\rangle})\otimes \Q$, which is of infinite dimension.

Over the last twenty years, this conjecture has been discussed in the literature in various forms, 
sometimes for all possible $\tilde K$ and sometimes only for subfields of $\bar\Q$,
often, but not always, assuming $A$ is an elliptic curve.
It is well known that in general $A(\bar K)$ has infinite rank (the result is attributed  in \cite{FJ} to a letter of Serre). The difficulty is to construct elements in $A(\bar K)$ in a way which controls
the action of the $\sigma_i$.
A survey of results on Conjecture~\ref{equiv} is given in \cite{IL-Survey}, but we just briefly mention a few approaches that have been fruitful.

\begin{enumerate}[leftmargin=10pt, labelwidth=0pt, labelsep=5pt, align=left]
\item[1.]\textbf{Arithmetic.} For elliptic curves over $\Q$, one can  take advantage of Heegner points. This has been done in various cases \cite{BreI, Ha, I2} and most recently, by  Im and  Choi
 \cite{CI},  for all $(\sigma_1,\ldots,\sigma_n)$ and elliptic curves over $\Q$ of analytic rank $\le 1$. In a somewhat different arithmetic direction,
Tim and Vladimir Dokchitser gave \cite{DD} a proof for all elliptic curves over number fields which are not totally imaginary, conditional on the Birch-Swinnerton-Dyer conjecture.

\item[2.]\textbf{Diophantine geometry.} The $K$-rational points on any quotient of $A^n$ by a finite group action can be lifted to points of  $A(\bar K)^n$ and therefore to points on $A(\bar K)$.
Hilbert irreducibility plays a role in proving that the points so constructed span an infinite-dimensional space.
The papers \cite{I1, ILR, Larsen} all use this method. A fairly general result in this direction \cite{IL-cyclic} is that in every dimension, Conjecture~\ref{equiv} holds for topologically cyclic fields $K$.

\item[3.]\textbf{Field arithmetic.}
A 2010 conjecture of Junker and Koenigsmann \cite{JK}
asserts that every characteristic zero field $K$ with finitely generated Galois group is \emph{ample},
meaning that every pointed non-singular curve over $K$ has infinitely many $K$-points (see \cite{BF} for a discussion of ample fields).
 Fehm and  Petersen prove \cite{FP} that if $K$ is {ample}, then every non-trivial abelian variety over $K$ has infinite rank.

\item[4.]\textbf{Additive combinatorics.} In \cite{IL-HJ}, we use Ramsey-theoretic ideas to prove Conjecture~\ref{equiv} whenever 
$A$ is an elliptic curve containing an affine open of the form 
\begin{equation}
\label{Jacobi}
v^2 = (u+c_0)(u+c_1)(u+c_2)(u+c_3),
\end{equation}
or more briefly, $v^2 = f(u)$, where $f$ is always understood in this paper to be a monic polynomial of degree $4$ which splits completely into linear factors.
This paper follows the method of \cite{IL-HJ} but eliminates this additional hypothesis, thereby proving the one-dimensional case of Conjecture \ref{equiv}.
\end{enumerate}

As an immediate corollary of Theorem~\ref{std}, we obtain the genus $1$ case of the Junker-Koenigsmann conjecture, since every pointed non-singular curve of genus $1$ over $K$
is a cofinite subset of an elliptic curve over $K$.

In \S2, we use results of Moon \cite{Moon} and Bays, Hart, and Pillay \cite{BHP} to prove Theorem~\ref{main}.
The papers we cite 
work in the setting of number fields, while we want a result which covers all finitely generated extensions of $\Q$,
so we make explicit exactly what arithmetic inputs they use and how we know that the same inputs are available in the generality 
we need.

The idea of our proof of Theorem~\ref{std} is as follows.  
Let $L$ be a finite Galois extension of $K$ such that
the $2$-torsion points of $A$ and at least one additional point are defined over $L$.
There exists an affine open subset of $A$ defined over $K$ which can be expressed
in the form \eqref{Jacobi} over $L$.
Let $\tilde K=\bar K^{\langle \sigma_1,\ldots,\sigma_n\rangle}$, and
let $\tilde L = \tilde K L$.
In \S3, we use a Ramsey-theoretic argument to prove that there exists a finite set 
$$\Sigma := \{\alpha_i t+ \beta_i\mid i=1,2,\ldots,h\}$$
of non-constant linear functions in $t$, with $\alpha_i,\beta_i\in L$, such that for all but finitely many
$t_0\in L$, there exists $u_0\in \Sigma(t_0)$ such that $\pm\sqrt{f(u_0)}\in \tilde L$. Choose a square root $v_0$, so $P := (u_0,v_0)\in A(\tilde L)$.
Applying trace, we obtain 
\begin{equation}
\label{Q-construction}
Q := \Tr_{\tilde L/\tilde K}P \in A(\tilde K).
\end{equation}
We would like to show that the set of points constructed in this way generates a group of infinite rank.
Each point $P=(u_0,v_0)$ is defined over some quadratic extension of $L$, so each point $Q$ in \eqref{Q-construction} is defined over
an extension of $K$ of bounded degree which is contained in $\tilde K$.

We remark that merely showing there exists some finite extension $\tilde L/\tilde K$ for which we can construct an infinite subset of $A(\tilde L)$  is very easy. Finding such an extension and such a subset for which one can prove that the traces generate an infinite rank group is much more difficult. What is good about the combinatorial construction mentioned above is
that the points it produces behave in some sense as if they were of positive density, and this offers hope of proving the claim of infinite rank.

To implement this idea, in \S4, we first use Chebotarev density to show that for any finite sequence $Q_1,Q_2,Q_3,\ldots, Q_N$ of points of the form 
\eqref{Q-construction},
there exists a specialization of $A$ to an elliptic curve $E/\F_p$ such that the subgroup of $E(\F_p)$ generated by the $Q_i$ in the sequence is of arbitrarily large index in $E(\F_p)$.
We then show that the proportion of elements of $E(\F_p)$ which can be represented by reducing points from \eqref{Q-construction}
is bounded away from $0$.
Together these two facts imply that the group generated by points from \eqref{Q-construction} cannot be generated by any finite set, which finishes the proof. This argument was suggested to us by the large sieve approach to proving quantitative Hilbert irreducibility. However,
there is a key difference in that in our
case we are trying to find a sufficiently good upper bound for ``density'' rather than proving that it is zero, and this can be achieved with a single large prime $p$; no actual sieving is needed.

It may seem surprising at first glance 
that a result asserting that $A(\tilde K)$ is large can be deduced from a theorem which says that $A(\Omega)$ is in some sense small.
We use Theorem~\ref{main} in order to show that the points $Q_i$  cannot all lie in a finite subspace.
The key is that, if they did, they would have to lie also in some virtually free abelian group $A(K(d))$ and would therefore have to lie in a finitely generated group,
contrary to the ``sieve''  argument.

It is worth noting that not every elliptic curve has an affine open of the form $v^2 = f(u)$ where $f$ is a monic quartic, as assumed in this introduction. In general, we must work
with twists of curves of this form, which adds some notational complexity but does not introduce real additional difficulty. 

The method of this paper ultimately depends on the existence of a non-trivial map from a hyperelliptic curve to $A$ (which in our case is an isomorphism). We may therefore hope that it may be applied to prove Conjecture~\ref{equiv} for all abelian surfaces. However, in dimension $\ge 3$, a very general abelian variety does not admit such a map \cite{Pirola}, so it seems a new idea will be needed to prove Conjecture~\ref{equiv} in general.

We would like to thank the referee for calling our attention to the references \cite{Moon} and \cite{BHP}, which led to
a substantial simplification of this paper.

\section{Silverman's Lemma and the Mordell-Weil-N\'eron Theorem}

This section is devoted to the proof of Theorem~\ref{main}.
 
We begin by assembling some known results about abelian varieties over fields finitely generated over $\Q$. Given
$A/K$ and a rational prime $\ell$, let $G_{\ell^\infty}$ denote the image of $G_K$
in the $\ell$-adic Galois representation $G_K\to \GL(T_\ell(A))$. Let $K_{\ell^\infty}$ denote the fixed field of $\bar K$
under the kernel of the Galois representation, i.e., the field extension of $K$ generated by the coordinates of all $\ell$-power torsion elements.
Let $K_{\tor}$ denote the compositum of all $K_{\ell^\infty}$, i.e., the field generated over $K$ by the coordinates of
all torsion points of $A$.

\begin{prop}
\label{facts}
With notations as above, we have the following facts:
\begin{enumerate}
\item The Mordell-Weil group $A(K)$ is finitely generated.
\item There exists a positive integer $d$ 
such that for all $\ell$, $G_{\ell^\infty}$ contains the homothety group $(\Z_\ell^\times)^d$.
\item There exists a finite extension $L/K$ such that the fields $K_{\ell^\infty}L$ are linearly disjoint over $L$.
\end{enumerate}
\end{prop}
 
 \begin{proof}
We have already mentioned that part (1) is N\'eron's generalization of the Mordell-Weil theorem.

Part (2)  was originally asserted by Serre \cite[2.2.5]{Serre84}. See \cite[\S2]{Serre86} for his proof in the number field case and,
e.g., \cite[Lemma 3.4]{JJ} for an explanation of how the general case follows from this case.

Part (3) is due to Serre \cite[Theorem 1]{Serre13} in the number field case.
In the general case it is due to Gajda and Petersen \cite[Corollary 1.2]{GP}.

 \end{proof}
 
 \begin{prop}
\label{K(d) torsion}
If $K$ is a finitely generated extension of $\Q$, $A/K$ is an abelian variety, and $d$ is a positive integer, then $A(K(d))_{\tor}$ is finite.
\end{prop}

\begin{proof}
This is due to Moon \cite[Proposition 2]{Moon} (and the following remark) in the number field case.
Given Proposition~\ref{facts}, the argument goes through more generally for finitely generated fields.
\end{proof}

\begin{prop}
\label{Ktor}
Let $K$ be a finitely generated extension of $\Q$ and $A/K$ an abelian variety of dimension $g\ge 1$.  
Then $A(K_{\tor})$ is the direct sum of its torsion subgroup $A(K_{\tor})_{\tor} \cong (\Q/\Z)^{2g}$
and a free abelian group. 
\end{prop}

\begin{proof}
This is due to Bays, Hart, and Pillay \cite[Lemma A.7]{BHP} in the number field case.
For the general case, in order to show that quotient $A(K_{\tor})/A(K_{\tor})_{\tor}$ is free, it suffices to
replace $K$ by $L$ and prove the same statement (since an arbitrary subgroup of a free group is free).
By fact (3), we may therefore assume the $K_{\ell^\infty}$ are linearly disjoint over $K$.
By fact (2), there exists $\gamma\in G_K$ whose image in each $\GL(T_\ell(A))$ for $\ell$ odd
is $2^d$. Therefore, $2^d-1$ kills $H^1(\Gal(K_{\tor}/K),A[\ell^m])$ for all $\ell>2$ and all $m\in \N$.
Also, $3^d-1$ kills $H^1(\Gal(K_{\tor}/K),A[2^m])$ for all $m\in \N$.
It follows that there exists a single positive integer which annihilates $H^1(\Gal(K_{\tor}/K),A[\ell^m])$ in all cases.
The proof now concludes exactly as in \cite{BHP}.
\end{proof}

\begin{lem}
\label{lim}
Let $V$ be a vector space over $\Q$, $\Lambda\subset V$ a subgroup and $V_1\subset V_2\subset V_3\subset \cdots$
a chain of finite-dimensional subspaces with union $V$.  If $\Lambda\cap V_i$ is free abelian for all $i$, then $\Lambda$ is free abelian.
\end{lem}

\begin{proof}
Let $\Lambda_0 = 0$ and
$\Lambda_n = \Lambda\cap V_n$ for $n>0$.  The quotient $\Lambda_{n+1}/\Lambda_n$ is finitely generated because $\Lambda_{n+1}$ is a free abelian subgroup of a finite-dimensional vector space over $\Q$ and therefore finitely generated. Since $\Lambda_{n+1}/\Lambda_n$ is also torsion free
(any element of $\Lambda_{n+1}$ which has a positive integer multiple in $\Lambda_n$ must lie in $V_n$ and therefore in $\Lambda_n$), it must be free.  It therefore lifts to a free subgroup $M_{n+1}$ of $\Lambda_{n+1}$.
It follows that $\Lambda = \bigoplus_{n\ge 0} M_{n+1}$ is free.
\end{proof}


If $W$ is any subspace of $A(\bar K)\otimes\Q$, we define $A(K)_W$ to be the intersection of $W$ with $A(K)\otimes 1$.

\begin{prop}
\label{after extension}
If $K$ is any finitely generated extension of $\Q$, $W = A(K)\otimes\Q$, and $d$ is a positive integer, then $A(K_{\tor}(d))_W$
is finitely generated.
\end{prop}

\begin{proof}
Let $\Lambda = A(K_{\tor})_W$ and $\Lambda' =A(K_{\tor}(d))_W$.  By Proposition~\ref{Ktor}, $\Lambda$ is a free abelian group of finite rank.
Since $\Lambda\subset \Lambda'\subset W$, it suffices to prove that $\Lambda'/\Lambda$ is annihilated by some positive integer.

Let $\ell > d$ be a prime. We claim that $\Lambda'/\Lambda$ has no element of order divisible by $\ell$. Indeed, suppose $Q$ belongs to $A(K_{\tor}(d))$ but not to $A(K_{\tor})$
and $\ell Q\in A(K_{\tor})$.  
There exists a finite sequence of Galois extensions $K_{\tor}^i$ of $K_{\tor}$ with $\Gal(K_{\tor}^i/K_{\tor})$ contained in the symmetric group $\SS_d$,
such that $Q$ is defined over $K_{\tor}^{[1,N]} := \prod_{i=1}^N K_{\tor}^i$. If $G = \Gal(K_{\tor}^{[1,N]}/K_{\tor})$, then $G$ embeds as a subgroup of $\SS_d^N$.

Consider the cohomology sequence of the short exact sequence of $G$-modules
\begin{equation}
\label{Q over l}
0\to A[\ell]\to A[\ell] + \Z Q \stackrel\ell\to \ell\Z Q\to 0.
\end{equation}
By definition, all torsion points of $A$ are defined over $K_{\tor}$, so $G$ acts trivially on $A[\ell]$,
so $H^1(G,A[\ell]) = \Hom(G,A[\ell]).$  However, $\ell$ is prime to the $(d!)^N$ and therefore to the order of $G$, so this cohomology group vanishes.
By the cohomology sequence of \eqref{Q over l}, $\ell Q\in H^0(G,\ell\Z Q)$ lifts to some element of 
$Q+A[\ell]$ which is $G$-invariant. Since the $G$-invariance of $Q+A[\ell]$ means that it belongs to $A(K_{\tor})$, we have $Q\in A(K_{\tor})$,
contrary to assumption.

Finally, for $\ell\le d$, it suffices to prove that there exists $m$ such that the order of an element of $\Lambda'/\Lambda$ cannot be divisible by $\ell^m$.
We choose $m$ such that $\ell^m$ does not divide $d!$. This implies that no element of $\SS_d^N$ has order divisible by $\ell^m$.
Proceeding as before, choosing $Q\in A(K_{\tor}(d))$ with $\ell^m Q\in A(K_{\tor})$ and $\ell^{m-1}Q \not\in A(K_{\tor})$ and considering the sequence of $G$-modules
$$0\to A[\ell^m]\to A[\ell^m] + \Z Q \stackrel{\ell^m}\to \ell^m\Z Q\to 0,$$
we claim that the image of every homomorphism $G\to A[\ell^m]$ lies in $\ell A[\ell^m]$, and this implies that $\ell^{m-1} Q$ lies in $A(K_{\tor})$, contrary to assumption.

\end{proof}

We can now prove Theorem~\ref{main}.

\begin{proof}
Let $K=K_1\subset K_2\subset K_3\subset \cdots$ be an infinite increasing chain of finite extensions of $K_1$ such that $\bigcup_i K_i = K(d)$. Such a chain exists since $K(d)$ is countable, and the extension of $K$ generated by every finite subset of $K(d)$ is finite.
Let $W_i  = A(K_i)\otimes \Q$.
For every $i$, every extension of $K$ of degree $\le d$ is contained in an extension of $(K_i)_{\tor}$ of degree $\le d$, 
so applying Proposition~\ref{after extension} to $K_i$, we have that
$A(K(d))_{W_i} \subset  A((K_i)_{\tor}(d))_{W_i}$ is finitely generated.  By Lemma~\ref{lim}, $A(K(d))\otimes 1$ is free abelian.

The short exact sequence
$$0\to A(K(d))_{\tor} \to A(K(d))\to A(K(d))\otimes 1\to 0$$
implies
$$A(K(d))\cong A(K(d))\otimes 1 \oplus A(K(d))_{\tor}.$$
By Proposition~\ref{K(d) torsion}, $A(K(d))_\tor$ is finite, so $A(K(d))$ is virtually free abelian.
\end{proof}

\section{The Hales-Jewett construction}

In this section, we recall how the Hales-Jewett Theorem \cite{HJ} can be used to construct points on $A(\tilde L)$ when $G_{\tilde L}$ is finitely generated
and $A(\tilde L)$ contains all points of $A[2]$ and at least one other point $P_0$. The reference for this material is \cite{IL-HJ} (note that Corollary 9 in that paper implicitly assumes the existence of $P_0$).

Let $g(x)$ be a monic cubic polynomial 
over $K$ with non-zero constant term, and $A$ be the elliptic curve over $K$
which contains the affine open curve $y^2 = g(x)$. Let $L/K$ be a finite Galois extension such that $A(L)$ contains full $2$-torsion, i.e.,
a field which contains the roots $e_1,e_2,e_3$ of $g(x)$.
A square root $y_0\in \bar K$ of $g(0) = -e_1e_2e_3$ determines a point $(0,y_0)$
on $A(K(\sqrt{-e_1e_2e_3}))$. Setting
$$u :=-\frac 1{x},\ z := \frac y{x^2},$$ 
we have
\begin{equation}
\label{twist}
\frac{-1}{e_1e_2e_3} z^2 = (u+c_0)(u+c_1)(u+c_2)(u+c_3),
\end{equation}
where
\begin{equation}
\label{c sub i}
c_0:=0,\ c_i := \frac 1{e_i},\ i=1,2,3.
\end{equation}
Thus, every elliptic curve over $K$ is a quadratic twist of a curve containing an affine open of the form \eqref{Jacobi};
if $A(K)$ contains any point which is not $2$-torsion, then the twist can be taken to be trivial, though the individual $c_i$ need
not all be defined over $K$.

Let $P$ be a point of the curve \eqref{Jacobi} defined over some extension field of $K$.
The map
$$\phi\colon (u,v) \mapsto  (u,y_0 v)$$
gives an isomorphism from this curve to the curve \eqref{twist}. Both curves are defined over $K$, but the isomorphism is only
defined over $K(y_0)$, so in general it does not commute with elements of the Galois group. However, it preserves the $u$-coordinate,
i.e., commutes with field automorphisms up to multiplication by $-1$.  Our strategy will be to find points $P$ 
of \eqref{Jacobi} defined over $\tilde L\supset K(y_0)$,
use $\phi$ to obtain $\tilde L$-points $\phi(P)$ of \eqref{twist}, which we can regard as $\tilde L$ points of $A$, and then
take $\Tr_{\tilde L/\tilde K} \phi(P)$ to obtain elements of $A(\tilde K)$. Note that the $u$-coordinates of $P$ and $\phi(P)$ are the
same, so $\Tr_{\tilde L/\tilde K} \phi(P)$ is the sum of points of \eqref{twist} of the form $(u_i,z_i)$, where the $u_i$ are all conjugate over
$\tilde K$.

In order to find a rich supply of $\tilde L$-points on \eqref{Jacobi}, we use Ramsey theory.
If $S$ is a finite set, a \emph{$\cC$-coloring} of a set $X$ will mean a function $X\to \cC$. A subset $Y\subset X$ is \emph{monochromatic} if the function is constant on $Y$.
As $G_{\tilde L}$ is finitely generated,
$${\tilde L}^\times/{\tilde L}^{\times 2}\cong \Hom(G_{\tilde L},\{\pm 1\})$$
is finite. We identify the set $\cC$ of colors with elements of $\tilde L^\times/\tilde L^{\times 2}$.
Multiplication by any element of $\tilde L^\times$ takes any monochromatic subset of $\tilde L^\times$
to another monochromatic subset.

If $A(\tilde L)$ contains $A[2]$ as well as some $P_0\not\in A[2]$,
then $A$ has an affine open set of the form $y^2 = (x-e_1)(x-e_2)(x-e_3)$ with $e_i\in \tilde L$ which contains some $(x_0,y_0)$,
with $y_0\neq 0$.
Translating $x$ by $-x_0$, we may assume $x_0=0$, so $y_0^2 = -e_1 e_2 e_3$.  The functions
$$u =-\frac 1{x},\ v = \frac y{y_0x^2},$$ 
satisfy equation \eqref{Jacobi},
where the $c_i$ are defined as in \eqref{c sub i}.

Any $4$-term sequence in $\tilde L^\times$ whose consecutive differences are $c_1-c_0$, $c_2-c_1$, and $c_3-c_2$ can be expressed
as $u+c_0,u+c_1,u+c_2,u+c_3$ for some value of $u$; if such a sequence is monochromatic, the product of its terms is a perfect square in $\tilde L$,
so the sequence determines a pair of points in $A(\tilde L)$.
More generally, any monochromatic $4$-term sequence in $\tilde L^\times$ whose consecutive differences are in ratio
$c_1-c_0:c_2-c_1:c_3-c_2$ determines a pair of points of $A(\tilde L)$.

Let $Z := \{0,1,2,3\}$. For any finite sequence $b_1,b_2,\ldots,b_N\in \Q$, and any $t_0\in \tilde L$, we define a function $\lambda_{t_0}\colon Z^N\to \tilde L$ by 
\begin{equation}
\label{L-coloring}
\lambda_{t_0}(i_1,\ldots,i_N) = t_0+\sum_{j=1}^N b_j c_{i_j}.
\end{equation}
We assume no non-empty subsequence of $b_1,\ldots,b_N$ sums to $0$. We regard this sequence as fixed, so for all but finitely many values of $t_0$,
the image of $\lambda_{t_0}$ lies in $\tilde L^\times$, and $\lambda_{t_0}$ therefore defines an $\cC$-coloring of $Z^N$.

A \emph{combinatorial line} in $Z^N$
is a subset of the form 
$$C = \{\vec v,\vec v+{\vec w},\vec v+2{\vec w},\vec v+3{\vec w}\}$$
with $\vec v=(v_1,\ldots,v_N)\in Z^N$, ${\vec w}=(w_1,\ldots,w_N)\in \{0,1\}^N$, and $w_i\neq 0$ for some $i$.  (Of course, when $w_i=1$, we must have $v_i=0$ in order for $\{\vec v,\vec v+{\vec w},\vec v+2{\vec w},\vec v+3{\vec w}\}$ to be contained in $Z^N$.)
Restricting $\lambda_{t_0}$ to the combinatorial line $C$, the sequence of values of \eqref{L-coloring} is
\begin{equation}\label{seq}t_0+r_C,\,t_0+r_C+s_Cc_1,\,t_0+r_C+s_Cc_2,\,t_0+r_C+s_Cc_3\end{equation}
where 
$$r_C= \sum_{\{j\mid w_j=0\}} b_jc_{v_j},\ s_C = \sum_{\{j\mid w_j=1\}}b_j\neq 0.$$
The differences between successive terms in the sequence \eqref{seq} are, as desired, proportional
to $c_1-c_0$, $c_2-c_1$, and $c_3-c_2$.
Indeed, if the sequence above is monochromatic, there exist a pair of points $(u,\pm z) = (u,\pm y_0 v)\in A(\tilde L)$, where
$$u = \frac{t_0+r_C}{s_C}.$$

A special case of the Hales-Jewett Theorem asserts

\begin{thm}
\label{HJ4}
For all finite sets $\cC$ there exists $N$ such that for every coloring of $Z^N$ by $\cC$, there exists a monochromatic combinatorial line.
\end{thm}

Defining $h$ to be the number $5^N-4^N$ of combinatorial lines, this theorem now implies:

\begin{thm}
\label{linear expressions}
There exists a collection of $h$ linear functions with coefficients in $\Q(c_1,c_2,c_3)\subset \tilde L$ such that for all but finitely many $t_0\in \tilde L$, at least one of these
functions, evaluated at $t_0$, gives the $u$-coordinate of a point on \eqref{Jacobi} defined over $\tilde L$
and therefore, since $y_0\in \tilde L$, a point of $A(\tilde L$.
\end{thm}

\section{The Chebotarev density theorem}

In this section, we use the Chebotarev density theorem for schemes over $\Z$ to complete the proof of Theorem~\ref{std}.

\begin{lem}
\label{multiquadratic}
Let $F(u)\in \F_p[u]$ be a monic polynomial of even degree $2d$ which is not the square of a polynomial in $\F_p[u]$.
For $i=1,2,\ldots,h$, let $\ell_i(t)=\alpha_i t + \beta_i$ be a non-constant linear function.  There exists $\epsilon>0$, depending only on the degree of $F(u)$ and on $h$
such that the number of $\xi\in \F_p$ such that $F(\ell_c(\xi))\in {\F_p^\times}^2$ for all $c$ is greater than $\epsilon p$ if $p$ is sufficiently large.
\end{lem}

\begin{proof}
We choose any $\epsilon < 2^{-h}$.
The genus of the projective non-singular curve $X$ with function field 
$$\F_p(u,\sqrt{F(\ell_1(u))},\ldots,\sqrt{F(\ell_h(u)})),$$
can be bounded, using the Riemann-Hurwitz theorem, in terms of $h$ and $\deg(F)$.  

We claim that $\F_p$ is algebraically closed in the function field of $X$. Indeed, the degree of this field over $\F_p(u)$ is a power of $2$, so the algebraic closure $\F_q$ of $\F_p$ in it must satisfy $[\F_q:\F_p] \in \{1,2,4,\ldots\}$. If this degree is greater than $1$, then $\F_q$ contains $\F_{p^2}$. We may assume $p>2$, so by Kummer theory, it suffices to prove that no $a\in \F_p^\times \setminus (\F_p^\times)^2$ is equivalent modulo $(\F_p(u)^\times)^2$ to an element in the multiplicative group generated by 
$$\{\alpha_i^{-2d}u^{-2d}F(\ell_i(u))\mid 1\le i\le h\}.$$ However, this group is contained in $(1+u^{-1}\F_p[[u^{-1}]])^\times$, so it contains no constant in $\F_p^\times$ which is a non-square.

By the Riemann hypothesis for geometrically irreducible curves over finite fields, the number of $\F_p$-points on $X$ is $(1+o(1))p$.
Now $X$ admits a cover of the projective line of degree $2^r \le 2^h$ such that the image $\xi$ of any point of $X(\F_p)$
in $\P^1(\F_p)$ is either $\infty$, a ramification point of the cover, or an element $\xi\in \F_p$ 
satisfying the desired quadratic residue conditions. 
The lemma follows.
\end{proof}

Now we consider all elliptic curves $E/\F_p$ where $E$ has an open affine curve $U$ of the form $v^2 = f(u)$, 
for a monic polynomial $f$ of degree $4$. We assume, for some positive integer $m$, that all the $m$-torsion points of $E$ are defined over $\F_p$.
We fix a rectangular array $\ell_{j,i}$ of non-constant linear functions in $\F_p[t]$, where $1\le j\le k$ and $1\le i\le h$.

\begin{prop}
\label{m and p}
Suppose $f$, $k$, $h$, and $\ell_{j,i}$ are as above, $m$ and $p$ are sufficiently large in terms of $k$ and $h$, and $p\nmid m$.
Then there exist $(\xi_1,\ldots,\xi_k)\in \F_p^k$ such that for every positive integer $i\le h$ and every $j\le k$, 
$f(\ell_{j,i}(\xi_j))\in (\F_p^\times)^2$ and for each $i$ and collection of points 
$$(P_1,\ldots,P_k) = ((\ell_{1,i}(\xi_1),v_1),\ldots,(\ell_{k,i}(\xi_k),v_k))\in U(\F_p),$$
every non-empty subsum of $P_1,\ldots,P_k$ lies in the complement of $m E(\F_p)$.
\end{prop}

\begin{proof}
By Lemma~\ref{multiquadratic}, the number of $k$-tuples $(\xi_1,\ldots, \xi_k)\in \F_p^k$ such that for each $j\le k$ and $i\le h$, we have
$f(\ell_{j,i}(\xi_j))\in (\F_p^\times)^2$, is at least $(\epsilon p)^k$,
where $\epsilon>0$ depends only on $h$. The number of elements $P$ in $m E(\F_p)$ is $m^{-2} |E(\F_p)| \le 2m^{-2} p$,
and the number of non-empty subsets $S \subset \{1,\ldots,k\}$ is $2^k-1$. Fixing both, the number of ordered $k$-tuples
$(P_1,\ldots,P_k)\in E(\F_p)^k$ such that $\sum_{s\in S} P_s = P$ is $|E(\F_p)|^{k-1} \le 2^{k-1} p^{k-1}$. We say a $k$-tuple $(u_1,\ldots,u_k)\in \F_p^k$ is \emph{bad} if it is the $k$-tuple of $u$-coordinates of such a  $k$-tuple $(P_1,\ldots,P_k)$. The number of bad $k$-tuples is 
less than $2m^{-2} 2^k 2^{k-1} p^k$. For fixed $i$, the number of $(\xi_1,\ldots,\xi_k)\in \F_p^k$ such that $(\ell_{1,i}(\xi_1),\ldots,\ell_{k,i}(\xi_k))$
is bad is likewise less than $m^{-2} 2^{2k} p^k$, and the number of $(\xi_1,\ldots,\xi_k)\in \F_p^k$ for which this condition holds for some $i$
is less than $h m^{-2} 2^{2k} p^k$. If $m > 2^k \sqrt h \epsilon^{-k/2}$, then this is strictly less than $(\epsilon p)^k$.
\end{proof}

Finally, we prove Theorem~\ref{std}. 

\begin{proof}
Given $A/{\tilde K}$, we can find a finitely generated extension $K$ of $\Q$ and an elliptic curve $A/K$ such that $A\times_{\Spec K} \Spec {\tilde K} \cong A$. Choose a finite Galois extension $L/K$ such that all the $2$-torsion of $A$ is rational over $L$ and there is at least one element of $A(L)$ which is not $2$-torsion. Let ${\tilde L} = {\tilde K}L$.  Its absolute Galois group, $G_{\tilde L}$, is of finite index in $G_{\tilde K}$ and therefore topologically finitely generated. 
There is an affine open subvariety of $A$ given by the equation \eqref{twist}.

Applying Theorem~\ref{linear expressions}, we obtain, for some $h$, a sequence of 
$h$ non-constant linear functions $\alpha_i t + \beta_i$ in $L[t]$ such that for all but finitely many $t_0\in L$,
one of these linear expressions, evaluated at $t_0$, gives the $u$-coordinate of a point $(u_i,v_i)$ on \eqref{Jacobi}, defined over some quadratic extension of $L$ which is
contained in $\tilde L$.  
We can express $\Tr_{{\tilde L}/{\tilde K}} \phi(u_i,v_i)\in A({\tilde K})$ as a sum 
\begin{equation}
\label{Trace}
\sum_{\sigma\in \Gal({\tilde L}/{\tilde K})} \sigma(u_i,y_0v_i).
\end{equation}
Note that there is a natural injective homomorphism $\Gal(\tilde L/\tilde K)\to \Gal(L/K)$, so we can regard the sum
\eqref{Trace} as being taken over a finite subgroup of $\Gal(L/K)$.

Our goal is to prove that the resulting set of points generates an infinite-dimensional subspace of $A({\tilde K})\otimes\Q$.
We know that all such points lie in $A(L(2))$, so by Theorem~\ref{main}, they generate a subgroup of a virtually free abelian group,
which is then also, necessarily, virtually free.  If the group they generate spans a finite-dimensional subspace of $A({\tilde K})\otimes\Q$,
then this group is finitely generated. 

Suppose, indeed, that
$Q_1,\ldots,Q_r$ denotes a finite sequence of elements of type \eqref{Trace} which generates the group of all such elements.
Let $k=[L:K]$.
We define $m$ as in Proposition~\ref{m and p}. 
We choose a finite Galois extension $M/K$ which contains $L$ and such that all points on $A(\bar K)$ of order $2m$
and all points $Q'\in A(\bar K)$ such that $mQ'\in \{Q_1,\ldots,Q_r\}$ are defined over $M$.

In order to show that a new element of type \eqref{Trace} is not in $\langle Q_1,\ldots,Q_r\rangle$, we seek a prime $p$
such that the mod $p$ reduction of this new element is not divisible by $m$ in $A(\F_p)$ but the mod $p$ reductions of
all $Q_i$ are $m$-divisible. To make sense of this, we need to spread $A$ out to an elliptic group scheme $\cA$ over a ring $R$
such that the generic fiber of $\cA$ is $A$.

Let $R_0\subset K$ denote the $\Z$-algebra generated by the elementary symmetric polynomials in the $c_i$.
We choose a finitely generated normal domain $R$ over $R_0$ with fraction field $K$, an abelian scheme $\cA$ over $\Spec R$,
and an affine open subscheme $\cU/\Spec R$ such that
$$\cU \cong  \Spec R[u,z]/(z^2 - F_A(u)),$$
where
$$F_A(u) := y_0^2u(u+c_1)(u+c_2)(u+c_3).$$
Let $S$ and $T$ denote the integral closure of $R$ in $L$ and $M$ respectively.
Thus, $R = S^{\Gal(L/K)} = T^{\Gal(M/K)}.$
Replacing $R$, $S$, and $T$ by
$R[a^{-1}]$, $S[a^{-1}]$, and $T[a^{-1}]$ respectively, for a suitable non-zero $a\in R$, we may assume $R$, $S$, and $T$ 
have a series of additional properties.

To begin with, we would like that after reduction modulo any maximal ideal
of odd residue characteristic, $\cU$ will be non-singular. We can accomplish this by
inverting the discriminant of $F_A(u)$.

As $\Z$ is universally Japanese \cite[0335]{Stacks}, it follows that $S$ and $T$
are finitely generated $R$-modules, so by generic freeness \cite[051S]{Stacks}, we may invert some element of $R$ to guarantee
that $S$ and $T$ are free $R$-modules of rank $k=[L:K]$ and $[M:K]$ respectively.

By definition, the set of points at which $\pi\colon\Spec S\to \Spec R$ fails to be smooth forms a closed set. As $S$ is a finitely generated $R$-algebra,
by Chevalley's constructibility theorem \cite[054J]{Stacks}, the image of the non-smooth locus by $\pi$ is constructible.
Since $L/K$ is separable, the generic point of $\Spec R$ is not in this image, and it follows that the image is contained in a proper closed
subset of $\Spec R$. Therefore, by inverting a suitable element of $R$, we may assume $S$ is a smooth $R$-algebra, which implies 
that it is finite \'etale, 
since it is module-finite.
By \cite[03SF]{Stacks}, $\Gal(L/K)$ acts simply transitively on each geometric fiber of $\Spec S\to \Spec R$.
The same argument applies for $\Spec T\to \Spec R$, and we may assume that this morphism is finite \'etale and 
each geometric fiber of $\Spec T\to \Spec R$ admits a simply transitive $\Gal(M/K)$-action.

We may assume that $y_0\in S$ and that all $Q_i$, all $m$-torsion points of $A$, and all $Q'$ such that $mQ'\in \{Q_1,\ldots,Q_r\}$ 
extend to sections over $T$.

We may assume that all $\alpha_i$ and $\beta_i$ lie in $S$ and that the $\alpha_i$ are invertible in $S$.

Suppose $\m$ is a maximal ideal of $R$ with residue field $\F_p$.  Then $T/\m T$ is an $\F_p$-algebra
of dimension $[M:K]$ and is a product of fields.
If at least one of these fields is $\F_p$, then we have a ring isomorphism
$T/\m T \cong \F_p^{[M:K]}$, which implies also the ring isomorphism $S/\m S\cong \F_p^k$.
In this  case, 
we denote by $\m_1,\ldots,\m_k$ the maximal ideals of $S$ lying over $\m$.  
Let $\rho_j\colon S\to \F_p$ denote reduction (mod $\m_j$). The Galois group $\Gal(L/K)$ acts simply transitively on the $\m_j$,
and we define $j(\sigma)$ so that $\m_{j(\sigma)} = \sigma^{-1}(\m_1)$.

Since $A$ is defined over $R$, $\rho_j(A)$ does not depend on $j$, and we call this common elliptic curve $E/\F_p$.
Likewise, $f(u) = \rho_j (u(u+c_1)(u+c_2)(u+c_3))$ does not depend on $j$; the affine curve $U\colon v^2 = \rho_j(y_0)^2f(u)$ over $\F_p$ 
is an affine open subvariety of
$E$. Note that this is isomorphic to the curve $v^2 = f(u)$ since $\rho_j(y_0^2) = \rho_j(y_0)^2 \in (\F_p^\times)^2$.

We define the non-constant affine linear functions $\ell_{j,i}(t) = \rho_j(\alpha_i t+\beta_i)$.
Then all points in $E(\bar{\F}_p)$ annihilated by $2m$
are defined over $\F_p$, and the reduction of each point $Q_i$ to $E(\F_p)$ lies in $m E(\F_p)$.

By Proposition~\ref{m and p}, there exists $(\xi_1,\ldots,\xi_k)\in \F_p^k$ such that for every positive integer $i\le h$ and collection of points 
$$(P_1,\ldots,P_k) = ((\ell_{1,i}(\xi_1),v_1),\ldots,(\ell_{k,i}(\xi_k),v_k))\in U(\F_p),$$
every non-empty subsum of $P_1,\ldots,P_k$ lies in the complement of $m E(\F_p)$.

We regard $H=\Gal(\tilde L/\tilde K)$ as a subgroup of $\Gal(L/K)$.

By the Chinese Remainder Theorem,
$$S/\m S\cong \prod_{j=1}^k S/\m_j \cong \F_p^k,$$
so there exist infinitely many elements of $S$ whose (mod $\m_j$) reduction is $\xi_j$ for all $j$.
Applying Theorem~\ref{linear expressions}, we obtain $t_0\in S$ such that for some $i$, $u_i = \alpha_i t_0 + \beta_i$
gives the $u$-coordinate of a point $(u_i,z_i)$ of  $A(\tilde L)$.  
Without loss of generality, we may assume $i=1$.
The $\Gal({\tilde L}/{\tilde K})$-orbit of $(u_1,z_1)$ consists of elements of the form
$(\sigma(u_1),z_\sigma)$, where $\sigma$ ranges over $\Gal({\tilde L}/{\tilde K})$.

These points are all defined over the extension $S'$ of $S$ generated by all $z_\sigma$.
Reducing a prime of $S'$ lying over (mod $\m_1$), we obtain a point of $E(\F_{p^2})$ whose $u$-coordinate is
$$\rho_1(\sigma(u_1)) = \rho_{j(\sigma)}(u_1) = \ell_{j(\sigma),1}(\xi_{j(\sigma)}),$$
so by Proposition~\ref{m and p}, this is a point on $E(\F_p)$, and, moreover, their 
sum does not lie in $m E(\F_p)$.
It follows that $\Tr_{{\tilde L}/{\tilde K}}(u_1,z_1)$ does not lie in the group generated by the $Q_i$.

It therefore suffices to show that there exist arbitrarily large primes $p$ for which $\Spec T$ has a point whose residue field is $\F_p$.
Any closed point $\Spec F$ on the generic fiber of $\Spec T$ extends to an $O_F[N^{-1}]$-point, where $O_F$ is the ring of integers of the number field $F$, and $N\in \Z$ is a positive integer. Thus $\Spec T$ contains an $\F_p$-point whenever $p$ is a sufficiently large prime which splits 
completely in $F$.
This happens for a positive density set of rational primes by the Chebotarev density theorem, and the theorem follows. 
\end{proof}

\end{document}